\documentclass[10pt, a4paper, oneside]{article}
\usepackage{amsmath, amsthm, amssymb, bm, graphicx, hyperref, mathrsfs, geometry}

\newtheorem{theorem}{Theorem}

\newtheorem{lemma}{Lemma}

\DeclareMathOperator{\esssup}{ess sup}
\DeclareMathOperator{\card}{card}
\DeclareMathOperator{\dens}{dens}

\DeclareMathOperator{\spa}{span}
\DeclareMathOperator{\sign}{sign}

\newcommand{\R}{\mathbb{R}}

\newcommand{\N}{\mathbb{N}}

\newcommand{\x}{\mathbf{x}}
\newcommand{\y}{\mathbf{y}}
\newcommand{\z}{\mathbf{z}}
\newcommand{\bu}{\mathbf{u}}

\newcommand{\M}{\mathcal{M}}

\title{\textbf{Optimal recovery of functions determined by second-order differential operators}}

\author{Bo Ling\footnote{School of mathematics and statistics, Xidian University, China(boling@xidian.edu.cn).},\,\, Yi Gu\footnote{School of mathematics and statistics, Yunnan University,  China(guyi@ynu.edu.cn, corresponding author).}}
\date{\today}
\begin{document}

\maketitle
	
\begin{abstract}
We study the optimal recovery problem for isotropic functions defined by second-order differential operators using both function  and gradient values. We derive the upper bound for $n$-th optimal error  with an explicit constant, which is independent of the specific form of the differential operators. Furthermore, for self-adjoint operators,  we obtain asymptotic exact results for the n-th optimal error.
\end{abstract}

\section{Introduction}
A fundamental task in optimal recovery is to determine how well  a function  in
a given space can be recovered from discrete information, which is usually given by function and/or derivative values(samples)\cite{MR584446,novak2010tractability}. The results of optimal error analysis provide a theoretical basis and benchmark for numerous numerical computation problems.

  Various results on the optimal recovery of functions have been documented, and we present a selection of the most pertinent references   to this study. In the one-dimensional  case, Koneichuk \cite{MR616672,MR616223} and Bojanov \cite{MR0393937} considered functions $f$ with bounded $r$-th derivatives and provided exact recovery results using function and derivative values at the knots. In the multivariate case, the optimal recovery of isotropic classes with $r$-th derivatives   was considered in \cite{MR2735421,2014Optimal,2018Optimal}, yielding exact or almost exact  results.
    
In this paper, we discuss  optimal recovery of isotropic functions $ W^{P(D)}_\infty(\Omega)$  defined on a convex domain $\Omega$ and  generated by a second-order differential operator $P(D)=D^2+p D+q$ using function  and gradient values at specific knots.The detailed definition of the function class $W^{P(D)}_\infty(\Omega)$  is provided in Section 2.1.  Let $\xi$ be a discrete subset of $\Omega,$ and   denote the information operator associated to $\xi$  by       
$I_{\xi} (f)= \{f(\x),  \nabla f(\x): \x\in \xi \} $  for $f\in W^{P(D)}_\infty(\Omega)$, where  $\nabla =(\frac{\partial }{\partial x_1}, \cdots  , \frac{\partial }{\partial x_d}  )$ is the gradient operator. 
A mapping  $\phi: I_\xi(W) \to C(\Omega) $ is called an algorithm.
The worst-case error of algorithm $\phi$ over the class $W^{P(D)}_\infty(\Omega)$ is defined by
\[R\left(W^{P(D)}_\infty(\Omega), \xi,\phi\right) := \sup_{f\in W^{P(D)}_\infty(\Omega)} \| f- \phi\circ I_\xi (f) \|_{\infty}. \]
The optimal error  using   function and gradient value at $\xi$ is defined as 
\begin{equation}\label{eq:R=radius} 
	R\left(W^{P(D)}_\infty(\Omega), \xi\right) = \inf_{\phi } \sup_{f\in W^{P(D)}_\infty(\Omega)} \|  f- \phi\circ I_\xi (f)    \|_{\infty}
\end{equation}
where   $\phi^\ast$  that attains the infimum on the right-hand side above  is called the optimal algorithm.

By the convexity and balance of $ W^{P(D)}_\infty$ and the linearity of information operator $I_\xi,$  it is deduced  from \cite{MR958691}(Section 5, Chapter 4)  that   the central algorithm is  optimal.
Consequently, the optimal error of the recovery problem  transforms into a multivariate extremal problem, as formulated below:  
\begin{equation}\label{ext-0}
	R\left(  W^{P(D)}_\infty,  \xi\right)=R\left(  W^{P(D)}_\infty,  \xi; \phi^c\right)=  \sup_{\substack{f\in   W^{P(D)}_\infty,\\ I_\xi(f)=0  }}\|f\|_{\infty}.
\end{equation}
Here $\phi^c$ denotes the central algorithm, defined as:
\[ \phi^c(w)(\x):=\frac{ 1}{2} \bigg\{ \sup_{\substack{f\in   W^{P(D)}_\infty,\\ I_\xi(f)=w}}  f(\x)+ \inf_{\substack{f\in  W^{P(D)}_\infty,\\ I_\xi(f)=w}} f(\x) \bigg\}. \]

The primary objective of this paper is to determine the n-th optimal error, denoted as
\begin{equation}\label{eq:app Rn}
	R_n( W^{P(D)}_\infty) := \inf_{\xi: \card(\xi) \leq n} R( W^{P(D)}_\infty,   \xi ),
\end{equation}	
and identify the optimal nodes set $\xi^\ast$ that achieves the infimum on the right-hand side of the equation.

  For $P(D)=D^2 $ and $D^3$  the  asymptotic exact results for $n$-th optimal error were shown  in \cite{MR2735421,2014Optimal}. We extend these results to   the general case of $P(D)=D^2+pD+q.$ Of particular interest is that the asymptotic behavior is  independent of the coefficients $p$ and $q.$

Firstly, we  derive an explicit   upper bound for $n$-th optimal error $R_n(W^{P(D)}_{\infty}(\Omega))$.     
\begin{theorem}\label{thm:n recovery upper}
	Let $\Omega\subset \R^d$ be a bounded and convex body, and $P(D)=D^2+p D+q$ be a second-order differential operator with constant coefficients $p,q\in \R$.   Then
	\begin{equation*}
		R_n(W^{P(D)}_\infty(\Omega)) 
		\leq     \frac{1}{4 }\left(\frac{\dens(d) \mu_d(\Omega)}{\nu_d n
		}  \right)^{2/d}(1+o(1)) \quad  \mbox{as} \,\, n\to \infty. 
	\end{equation*}
	where  $\mu_d(\Omega)$ is the volume of $\Omega$, $\nu_d$ is the volume of unit ball in $\R^d,$ and the constant $\dens(d)$ is the least density of sphere covering of $\R^d$ as defined in (\ref{eq:dens(d)}). 
\end{theorem}

Furthermore, if the differential operator is self-adjoint, i.e.,  the first-order coefficient  $p=0$, we can derive the exact result  for $n$-th optimal error   $R_n(W^{P(D)}_{\infty}(\Omega)).$ 
\begin{theorem}\label{thm:n recovery exact}
	Let $\Omega\subset \R^d$ be a bounded and convex body, and $P(D)=D^2+ q$ be  a second-order differential operator with constant coefficients $ q\in \R$.   Then
	\begin{equation*}
		R_n(W^{P(D)}_\infty(\Omega)) 
		=  \frac{1}{4 }\left(\frac{\dens(d) \mu_d(\Omega)}{\nu_d n
		}  \right)^{2/d}(1+o(1)) \quad  \mbox{as} \,\, n\to \infty. 
	\end{equation*}
	where  $\mu_d(\Omega)$ is the volume of $\Omega,$  $\nu_d$ is the volume of unit ball in $\R^d,$ and the constant $\dens(d)$ is the least density of sphere covering of $\R^d$ as defined in (\ref{eq:dens(d)}). 
\end{theorem}

We organize this paper as follows. Section 2 introduces   necessary notation and presents relevant preliminary results concerning optimal sphere coverings, Green's functions and univariate extremum problems.  Theorems 1  and 2 are proved in Section 3 and 4, respectively.

\section{Preliminaries}

\subsection{Notations}
 Denote $B(\x,r)$ the open ball  and  $B[\x,r]$ the closed ball   of radius $r$ centered at $\x.$    
Given   a set $E \subset \R^d$,  we denote by $\mu_d(E)$ the Lebesgue measure of $E$ in $\R^d$, and $\card(E)$ the cardinality of  the finite set $E$. Furthermore, we define
$e(\x,E)=\inf_{\y\in E} |\x-\y|$ as the distance from a point $\x\in \R^d$ to a set $E\subset \R^d,$  and $e(X,E)=\sup_{x\in X}  e(\x,E)$  
as the one-sided Hausdorff distance of $X\subset \R^d$ to $E\subset \R^d.$

Let $P(D)=D^2+p D+q$ be a second-order differential operator with constant coefficients $p,q\in \R$.  Let $\Omega  \subset \mathbb R^d$  denote a convex body, i.e., a convex closed set with non-empty interior. Let $C^k(\Omega)$ be the space
of $k$-times  continuously differentiable functions on $\Omega$.   Let  $W^{P(D)}_\infty(\Omega)$ be the set of all the functions $f: \Omega\rightarrow \R$   such that the directional derivative $\frac{\partial^2 f}{\partial \mathbf u^2}$ exists  for every unit vector $\mathbf u\in \mathbb R^d$  inside $\Omega$ in the generalized sense and $\| P(\frac{\partial   }{\partial \mathbf u } ) f  \|_\infty \leq 1.$ Here,  the norm $\|\cdot\|_\infty$ denotes the essential supremum norm, i.e.,
$\|f \|_\infty:=\esssup_{x\in \Omega} |f(x)|.$ 

The existence of the generalized derivative $\frac{\partial^2 f}{\partial \mathbf u^2}$ implies that $f\in C^{1} (\Omega)$ and, for almost every straight line $l$ parallel to $\mathbf u$ and passing through the interior of $\Omega,$  the restriction of $\frac{\partial  f}{\partial \mathbf u}$ to the intersection of $l$ with  $\Omega$ is locally absolutely continuous and $\frac{\partial^2 f}{\partial \mathbf u^2}$ is measurable. 
Specially, for the univariate case,   $W^{P(D)}_\infty[a,b]$ denotes the set of all functions $h$ such that $h'$ is absolutely continuous and $\| P(\frac{d   }{d t} )  h (t)\|_{L_\infty[a,b]} \leq 1.$

\subsection{Optimal Sphere Covering  }

The sphere covering problem focuses on determining the most economical ways to cover $\R^d$ with equally sized balls. We will demonstrate that this problem is intricately linked to the optimal recovery problem we previously considered.  
A  covering  of $\R^d$ comprises a countable collection $\M$ of   balls with identical radii, whose union equals $\R^d.$ 
The   covering density of $\M$ is defined as follows:
\begin{equation*}
	\dens \M := \liminf_{R \to \infty}  \frac{\sum_{B\in \M}  \mu_d(B\cap [-R,R]^d  )  }{ \mu_d([-R,R]^d) }.
\end{equation*}
The quantity
\begin{equation}\label{eq:dens(d)} 
	\dens(d):= \inf\{ \dens(\M): \M \mbox{ covers}  \R^d   \}
\end{equation}
is termed the least density of sphere covering of $\R^d$, and 	
the covering $\M^\ast$ that attains the infimum in Equation (\ref{eq:dens(d)})
is referred to as the optiaml sphere covering of $\R^d.$  The set of centers of the optimal sphere covering is denoted by   $ \xi^\ast=\xi^{\ast,d}.$

Identifying the optimal sphere covering of $\R^d$
is an intriguing and significant problem. 
Kershner was the first to consider this problem, demonstrating in \cite{MR0000043} that the hexagonal lattice provides the optimal sphere covering in the planar case. Since then, the lattice covering problem (where the centers of spheres form a lattice) has been solved up to dimension 5 \cite{MR1662447}. However, for dimensions $d\ge 6$, even the optimal lattice covering remains unknown.

For   a bounded subset $\Omega$ of $\R^d,$ 
define
\begin{equation}\label{eq:e-n for omega}
	e_n(\Omega):= \inf \{ e(\Omega,\xi):   \card \xi\leq n, \xi\subset \R^d\},
\end{equation}
which is called the  $n$-covering radius  of $\Omega.$ If 
$\xi^\ast$  attains the
infimum on the right-hand side of (\ref{eq:e-n for omega}), we refer to $\xi^\ast$ as an
$n$-centers  for $\Omega.$    Lemma
10.3\cite{MR1764176} guarantees  the existence of $n$-centers  for a non-empty compact
set $\Omega.$  Consequently, $n$-centers exist for  a convex body $\Omega \subset \R^d$.

Given that
\[  e(\Omega,\xi)=\max_{\y\in \Omega} \min_{\x\in \xi}
|\y-\x|=\min\{\lambda\geq 0: \Omega\subset \bigcup_{\x\in \xi} B(\x,\lambda) \}, \]
it implies    
\[e_n(\Omega ) = \inf\{\lambda:   \Omega  \subset \bigcup_{i=1}^n B[\x_i,\lambda], \x_1,\ldots,\x_n \in \R^d  \}.\]
Thus, seeking   $n$-centers for  $\Omega$  is equivalent to solving the geometric problem of finding the optimal sphere covering of $\Omega$.

Kolmogorov and  Tihomirov  presented the asymptotic exact values of $e_n(\Omega)$ in   Theorem IX\cite{MR0112032} as follows. 
\begin{lemma}\label{lem:e-n}
	For every compact set $\Omega\subset \R^d$ with $\mu_d(\Omega)>0$ and $\mu_d(\partial \Omega)=0,$ it holds that
	\[  e_n(\Omega) =  \bigg(\frac{\dens(d) \mu_d(\Omega)}{n \nu_d } \bigg)^{1/d}  (1+o(1)), \; \mbox{ as } n\to \infty, \]
	where  $\nu_d$ is the volume of unit ball in $\R^d$.
\end{lemma}
We will observe that the  asymptotic exact value of the $n$-th optimal error $R_n(W^{P(D)}_\infty(\Omega))$ is closely related to $e_n(\Omega).$

\subsection{Green's Functions of Second-order Differential Equations }
Let $P(D)=D^2+p D+q$  be a second-order differential operator  with constant coefficients $p,q\in \R.$ Differential operators $P(D)$ are categorized into three types based on their characteristic roots: (1) $(D-\alpha)^2$ with $\alpha\in \R;$ (2) $ (D-\alpha)(D-\beta)$ with $\alpha<\beta, \alpha,\beta\in \R;$ (3) $ D^2 +2\alpha D +(\alpha^2 +\beta^2)$ with $\alpha\in \R$ and $\beta>0.$

The solution to the boundary value problem
\begin{align*}
	\begin{cases}
		P(\frac{d}{dt}) f (t)= \varphi(t),  \quad t\in [0,a],\\
		f(a)=0, f'(a)=0,	 	
	\end{cases}
\end{align*}  
is    given by     $f(t)= \int_0^a  G(t,\tau) \varphi(\tau) d\tau,$ where $ G(t,\tau)$ is the  Green's function  for  the boundary value problem.   We can express  $ G(t,\tau) = g\left((\tau-t)_+\right),$ where $(t)_+=\max\{t,0\} $  and $g(t)$ is defined  by 
\begin{equation}\label{eq-g}
	g ( t)=\begin{cases}
		t e^{-\alpha t}, & P(D)=(D-\alpha I)^2, \\	
		\frac{1}{\beta-\alpha} \left[e^{-\alpha t} -e^{-\beta t} \right], &  P(D)= (D-\alpha I) (D-\beta I),\\ 
		\frac{1}{\beta }  e^{-\alpha t}  \sin\beta t , & P(D)=P(D)= D^2 +2\alpha D +(\alpha^2 +\beta^2)I. 	
	\end{cases}	    
\end{equation}

Its derivative  is 
\begin{equation*}
	g'(t)=\begin{cases}
		[1-\alpha t]e^{-\alpha t}, & P(D)=(D-\alpha I)^2, \\	
		\frac{1}{\beta-\alpha} \left[\beta e^{-\beta t} -\alpha e^{-\alpha t} \right], &  P(D)= (D-\alpha I) (D-\beta I),\\ 
		\frac{1}{\beta }  e^{-\alpha t}  [\beta \cos \beta t- \alpha\sin\beta t ], & P(D)=P(D)= D^2 +2\alpha D +(\alpha^2 +\beta^2)I. 	
	\end{cases}	    
\end{equation*}
It is straightforward to verify  that $g(0)=0, g'(0)=1$ and $g$ is strictly increasing in the interval $[0,\delta)$  where
\begin{equation}\label{eq-delta}
	\delta =\begin{cases}
		+\infty,   & P(D)=(D-\alpha I)^2, \alpha\leq 0, \\
		\frac{1}{\alpha},   & P(D)=(D-\alpha I)^2, \alpha> 0, \\	
		+\infty, &  P(D)= (D-\alpha I) (D-\beta I), \beta\leq 0,\\	
		\frac{1}{\beta}, &  P(D)= (D-\alpha I) (D-\beta I), \beta> 0,\\ 
		\frac{1}{\beta}\arctan \frac{\beta}{\alpha}, & P(D)=  D^2 +2\alpha D +(\alpha^2 +\beta^2)I,\alpha>0,\\
		\frac{1}{\beta}\frac{\pi}{2}, 	& P(D)= D^2 +2\alpha D +(\alpha^2 +\beta^2)I,\alpha\leq 0. 
	\end{cases}	    
\end{equation}

For the convenience of later discussion, let  $G(t)$  be the antiderivative of $g(t)$ with $G(0)=0,$ that is,  
\begin{equation}\label{eq-G}
	G ( t)=\begin{cases}
		-\frac{(\alpha t+1)}{\alpha^2}  e^{-\alpha t}  + \frac{1}{\alpha^2}  (\alpha \neq 0) \quad \mbox{or} \quad \frac{t^2}{2} (\alpha = 0), & P(D)=(D-\alpha I)^2,   \\	
		\frac{1}{\beta-\alpha} \left[\frac{e^{-\beta t}}{\beta} -\frac{e^{-\alpha t}}{\alpha} \right]+\frac{1}{\alpha \beta}, &  P(D)= (D-\alpha I) (D-\beta I),\\ 
		\frac{1}{\beta } \frac{-\alpha \sin \beta t -\beta \cos\beta t}{\alpha^2+\beta^2} e^{-\alpha t} +\frac{1}{\alpha^2+\beta^2}   , & P(D)=P(D)= D^2 +2\alpha D +(\alpha^2 +\beta^2)I. 	
	\end{cases}	    
\end{equation} 
It is straightforward to verify that $G(0)=G'(0)=0, G''(0)=1,$ and $G$ is strictly increasing and  convex  on $[0,\delta).$

\subsection{Some Univariate Extremal Results}
In the sequel, we always convert the multivariate extremal problem (\ref{ext-0})   into a univariate  one. The following two univariate extremal problems are involved:
\[ \sup\{ |h(0)| :  h\in W^{P(D)}_{\infty}[0,a],  h(a)=h'(a)=0   \},\] 
and
\[ 	\sup\{ |h(0)| :  h\in W^{P(D)}_{\infty}[0,a],  h(a)=h'(a)=0, h'(0)=0  \}.
\]

For the three types of differential opertors, the above extraml problems   are solved   in the following lemma. 

\begin{lemma}\label{lemma1}
	For  $a\in (0,\delta),$ where $\delta$ is defined as (\ref{eq-delta}),  it holds that
	\begin{equation}\label{ext-1}
		\sup\left\{|h(0)|:  \|P(D)h\|_\infty\leq 1,h(a) =h'(a)=0\right\}    = G (a),  
	\end{equation} 
	\begin{equation}\label{ext-2}
		\sup\left\{|h(0)|: \|P(D)h\|_\infty\leq 1,h(a) = h'(a)=  h'(0)=0\right\}   	   =    G(a)-2G(t_0),   	
	\end{equation} 
	where $G $ and  $g$  is defined in (\ref{eq-G}) and (\ref{eq-g}), respectively, and $ t_0 = g^{-1}\left(\frac{1}{2}g(a)\right)  $	with $g^{-1}$ being the inverse function of $g.$
 Furthermore, 	$G(a) -  2 G\left( g^{ -1 } (\frac{1}{2}g(a) )  \right)  $ is increasing with respect to the variable $a$. 
\end{lemma}

\begin{proof}
	The proof procedures for the three cases are similar, so  we only prove the first case   in detail  and provide a brief explanation for the other two cases.
	
	For $P(D)=(D-\alpha I)^2,$   the solution to the boundary value problem 
	$  P(\frac{d}{dt}) h(t) =\varphi(t),   h(a)=h'(a)=0      $ 
	is   
	$  h(t) = \int_0^a  g\left((\tau -t)_+ \right) \varphi(\tau)  d\tau.    $
	Thus we have $h(0)=\int_0^a  g(\tau) \varphi(\tau)  d\tau.$  
	
	The solution to the extremal problem  (\ref{ext-1})  is 
	\begin{equation*} 
		\sup_{\|P(\frac{d}{dt})h\|_\infty\leq 1;h(a) =h'(a)=0}  |h(0)|   = \sup_{\|\varphi\|_\infty\leq 1}    \left|  \int_0^a   g(\tau) \varphi(\tau)  d\tau    \right| 	\nonumber 
	 =     \|  g \|_1 = G(a).
	\end{equation*}
	where the last equality holds because  $g$ is non-negative and   $G(0)=0$.

	We now turn to the second extremal problem (\ref{ext-2}). The additional condition $h'(0)=0$
	is converted to
	$ \int_0^a g'(\tau)  \varphi(\tau )  d\tau=0.  $ 
	By the dual theory of extremal problem (Proposition 1.4.1\cite{MR1124406}), we have
	\begin{align}\label{eq-l1}
		\sup_{\|P(\frac{d}{dt})h\|_\infty\leq 1;h(a) =h'(a)=0; h'(0)=0}  |h(0)| &= \sup_{\|\varphi\|_\infty\leq 1;  h'(0)=0}    |  \int_0^a    g(\tau) \varphi(\tau)  d\tau    |  \nonumber \\
		 &= \sup_{\|\varphi\|_\infty\leq 1;   \varphi \perp g'}    |  \int_0^a    g(\tau) \varphi(\tau)  d\tau    |  \nonumber \\
		&= \min_{c} \|     g(\tau) - c  \cdot  g'(\tau)   \|_1, 
	\end{align}
	 where the last term is  the best approximation of $ g(\tau) $ from the one-dimensional space $ \spa\{g'(\tau) \}$ in the $L_1$ norm.

Since $g$ is a  strictly increasing function on $[0,a]$ with $g(0)=0,$    there exists a unique $t_0\in (0,a)$ such that $\int_0^a   \sign(\tau -t_0) \cdot   g'(\tau)  d\tau =0,$   where  $ t_0 = g^{-1} \left( \frac{1}{2} g(a)\right)$ and $\sign t$ is the sign function with $\sign t=1$ for $t\geq 0$ and $\sign t=-1$ for $t<0.$ Let $\varphi_0(\tau):=  \sign (\tau-t_0), $ so that $ \varphi_0\perp g'.$

	It is easy to verify that  $\{    g'(\tau) ,     g(\tau)   \}$ forms a Chebyshev system on the interval $[0,a]$.  By Theorem 2.4-5 in \cite{Sun}, the best $L_1$ approximant $c_0 g'$  in (\ref{eq-l1}) exists,  and   $  g - c_0     g'$ changes signs at $t_0,$ i.e.,  $  \sign  [  g(\tau) - c_0    g'(\tau)     ]=  \sign( \tau-t_0).$ Therefore, 
\begin{equation}\label{ba}
\min_{c\in \R} \|     g  - c  \cdot  g'    \|_1 = \int_0^a | g(\tau) - c_0    g'(\tau)  | d\tau = \int_0^a \left(g(\tau) - c_0    g'(\tau)  \right) \varphi_0(\tau) d\tau =\int_0^a g(\tau) \varphi_0(\tau) d\tau.  
\end{equation} 
	
Let $\tilde h (t) := \int_0^a g\left((\tau -t)_+ \right) \varphi_0(\tau)  d\tau.$ 		
Then$\|P(\frac{d}{dt}) \tilde h\|_\infty= \|  \varphi_0\|_\infty \leq 1,    \tilde h (a)= \tilde h' (a)=0,$ and $ \tilde h'(0)=0.$ By the previous results (\ref{eq-l1}) and (\ref{ba}),  
the solution to the extremal problem (\ref{ext-2}) is 	
\begin{equation*}\label{extremal2.1}
	\sup_{\|P(D)h\|_\infty\leq 1;h(a) =h'(a)=0; h'(0)=0}  |h(0)| 	 = \tilde h(0) 
	=    G(a)-2G(t_0).   	
\end{equation*}
This completes the proof of Lemma \ref{lemma1}.  
\end{proof}

Moreover,   the explicit form of  $\tilde h(t)$ is given by
\begin{equation}\label{eq-h0}
	\tilde h(t)=\begin{cases}
		G(a-t)-2G(t_0-t) ,& 0\leq t\leq t_0,\\  	
		G(a-t), & t_0<t\leq a, \\
		0,& t>a.  	
	\end{cases} 	
\end{equation} 
If $P(D)=D^2,$ we obtain  $g(t)=t,$       $\displaystyle G(t)=\frac{t^2}{2},$ and  $\displaystyle t_0=\frac{a}{2}.$  Consequently, the right-hand side of (\ref{ext-1}) and (\ref{ext-2}) equals $\frac{a^2}{2}$ and $\frac{a^2}{4}$ respectively. 

In general, $t_0=g^{-1}\left(\frac{1}{2}g(a)\right)$ is related to the function $g$.  However, it holds that  $\displaystyle\lim\limits_{a\to 0} \frac{ t_0}{a}=\frac{1}{2}$ due to $g'(0)=1.$ Furthermore, as $a\to 0^+,$ we have
\begin{equation}\label{asy G}	
G(a)=\frac{a^2}{2} (1+o(1) ), \quad 	G(a)-2G(t_0) =\frac{a^2}{4} (1+o(1) ).
\end{equation}   
It should be emphasized that  this exact asymptotic result is independent of the types of  the second differential operator $P(D).$

In Section 4, some explicit forms of $\tilde h$ are needed for specific differential operators. For $P(D)=D^2,$ 
\[ \tilde h (t)=\begin{cases}
	\frac{a^2}{2}-\frac{t^2}{2}, & t\in [0,\frac{a}{2}],\\
	\frac{(a-t)^2}{2},& t\in (\frac{a}{2},a].	
\end{cases}   \]

For $P(D)=D^2-\beta^2,$ 
\begin{equation}\label{eq-h0-2}
\tilde h (t)=\begin{cases}
	\frac{1}{\beta^2} \left[ \cosh \beta(a-t) -2 \cosh \beta(t_0-t)  +1  \right], & t\in [0,t_0],\\
	\frac{1}{\beta^2} \left[ \cosh \beta(a-t) -1  \right],& t\in (t_0,a],	
\end{cases}    	 
\end{equation} 
where $t_0$ satisfies $ 2\sinh \beta t_0 =\sinh \beta a.$

For $P(D)=D^2+\beta^2,$ 
\begin{equation}\label{eq-h0-3}  \tilde h (t)=\begin{cases}
	\frac{1}{\beta^2} \left[ -\cos  \beta(a-t) +2 \cos  \beta(t_0-t) -1  \right], & t\in [0,t_0],\\
	\frac{1}{\beta^2} \left[   1- \cos \beta(a-t)  \right],& t\in (t_0,a],	
\end{cases}   
\end{equation} 
where $t_0$ satisfies $ 2\sin  \beta t_0 =\sin  \beta a.$

\section{Proof of the Explicit Upper Bound}

The following lemma demonstrates that   the restriction of a function  $f\in W^{P(D)}_{\infty}(\Omega)$  to any straight line  remains a univariate function in   $W^{P(D)}_{\infty}[0,a].$
\begin{lemma}\label{lem:multi-uni}
Let $\Omega\subset \R^d$ be a convex body, and $P(D)=D^2+p D+q$ be a second-order differential operator with constant coefficients $p,q\in \R$. 
	If $f\in W^{P(D)}_{\infty}(\Omega),$ then for any two distinct points $\x,\y\in \Omega,$ the function
	\[g(t)=f(\x+t \frac{\y-\x}{|\y-\x|}) \]
	belongs to $W^{P(D)}_{\infty}[0,a],$ where $a=|\y-\x|.$
\end{lemma}
\begin{proof}
	Let $\bu= \frac{\y-\x}{|\y-\x|}.$  Since  $f\in C^1(\Omega),$ it follows that
	\begin{equation*} 
		g'(t) =(\bu\cdot \nabla)  f(\x+ t \mathbf u ) =\frac{\partial }{\partial \mathbf u} f(\x+ t \mathbf u), \; t\in [0,a].
	\end{equation*}
Furthermore, $g'(t)$ is absolutely continuous by the definition of $ W^{P(D)}_{\infty}(\Omega) $ and 
	\[ g''(t)= \frac{\partial^2 }{\partial^2 \mathbf u} f(\x+ t \mathbf u),\quad  \mbox{almost everywhere } \quad  t\in[0,a].      \]
	Therefore,  $   P(\frac{d}{dt})g(t) =  P(\frac{\partial }{\partial \mathbf u}) f(\x+ t \mathbf u)  $
	and $\|  P(\frac{d}{dt})g  \|_{L_\infty[0,a]}\leq  \|P(\frac{\partial }{\partial \mathbf u}) f \|_{\infty} \leq 1,$  
	which implies that $g\in W^{P(D)}_{\infty}[0,a].$
\end{proof}	

By lemma \ref{lem:multi-uni}, we can   utilize univariate extremal results to derive an upper bound of multivariate    extremal problem  in (\ref{ext-0}).

\begin{theorem}\label{thm:upper-recovery}
	Let  $\Omega \subset \R^d$ be a bounded and convex body, and $\xi\subset \Omega$ a finite set of nodes such that   $e(\Omega,\xi)<\delta,$ where $\delta$ is defined in (\ref{eq-delta}). Let $P(D)=D^2+p D+q$  be a second-order differential operator with constant coefficients $p,q\in \R$.  Then, the central algorithm $\phi^c$ is  the optimal algorithm for recovery problem $R(W^{P(D)}_\infty(\Omega),\xi)$ and
	\begin{align*}
		R(W^{P(D)}_\infty(\Omega), \xi) &=R(W^{P(D)}_\infty(\Omega), \xi;\phi^c)    \nonumber\\ &\leq \frac 12  \max\bigg\{ G\left(  e (  \Omega,\xi)  \right) - 2G\left(  g^{-1}\left(\frac{g(e (  \Omega,\xi))}{2}\right)  \right) ,G\left(  e (\partial \Omega,\xi)  \right) \bigg\};
	\end{align*}
	if further $G(  e (\partial \Omega,\xi)  ) \leq G(  e (  \Omega,\xi)  ) - 2G\left(  g^{-1}\left(\frac{g(e (  \Omega,\xi))}{2}\right)  \right),$ it holds
	\begin{equation}\label{eq:xi recovery}
		R(W^{P(D)}_\infty(\Omega), \xi) \leq \frac 12G(  e (  \Omega,\xi)  ) -  G\left(  g^{-1}\left(\frac{g(e (  \Omega,\xi))}{2}\right)  \right),
	\end{equation}
	where functions $g$ and $G$  are defined in (\ref{eq-g}) and (\ref{eq-G}), respectively.
\end{theorem}

\begin{proof}
	The optimality of $\phi^c$ is derived  from
	(\ref{eq:R=radius}). For any  $f\in W^{P(D)}_{\infty}(\Omega)$ with information $I_\xi(f)=0,$
	let $\x_0\in \Omega$ be a maximum point of $|f(\x)|$ on $\Omega.$ Let
	$\y$ be a closest point in $\xi$ to $\x_0$ and set $a:=e(\x_0,\xi)=|\x_0-\y|$. First, if $\x_0$ lies
	in the interior of $\Omega,$ then $\nabla(f)(\x_0)=0.$ By Lemma
	\ref{lem:multi-uni},
	\[h(t)=f(\x_0+ t\frac{\y-\x_0}{a}), \; t\in [0,a]  \]
	belongs to $W^{P(D)}_\infty[0,a].$ Since $\nabla(f)(\x_0)=0$ and $I_\xi(f)=0,$ we have $h'(0)=0$ and $h(a)=h'(a) =0.$
	Thus, by (\ref{ext-2}) in Lemma  \ref{lem:multi-uni}, it holds
	\begin{align}\label{eq:interior-upper1}
		\|f\|_\Omega=|f(\x_0)|&=|h(0)|\leq G(a)-2G\left(g^{-1}\left(\frac{g(a)}{2}\right)\right).
	\end{align}

	Second, if $\x_0$ is on the boundary of $\Omega,$ by  (\ref{ext-1}) in Lemma  \ref{lem:multi-uni}, we have
	\begin{align}\label{eq:boundary-upper}
		\|f\|_\Omega=|f(\x_0)|&=|h(0)|\leq G(a).
	\end{align}
	By (\ref{eq:interior-upper1}),  (\ref{eq:boundary-upper}) and (\ref{ext-0}), it follows that
	\begin{equation*}
		R(W^{P(D)}_{\infty}(\Omega) ,\xi)=\sup_{\substack{f\in W^{P(D)}_\infty(\Omega); \\I_\xi(f)=0}  }\|f\|_\Omega\leq\frac 12  \max\bigg\{ G\left(  e (  \Omega,\xi)  \right) - 2G\left(  g^{-1}\left(\frac{g(e (  \Omega,\xi))}{2}\right)  \right) ,G\left(  e (\partial \Omega,\xi)  \right) \bigg\}.
	\end{equation*}	
The proof of the remaining (\ref{eq:xi recovery}) is straightforward.
\end{proof}

We now proceed to the proof of Theorem \ref{thm:n recovery upper} now.
\begin{proof} 
	
Firstly, we   construct a set $\xi^\ast_n$ that serves as an almost n-centers approximation for $\Omega$, subject to the additional condition:  $G(  e (\partial \Omega,\xi)  ) \leq G(  e (  \Omega,\xi)  ) - 2G\left(  g^{-1}\left(\frac{g(e (  \Omega,\xi))}{2}\right)  \right).$

For $h>0,$ let $D_h:=D_h(\Omega):=\{\x:  e(\x,\partial \Omega)<h \}$ denote the $h$-neighborhood of $\partial \Omega.$ Let $\theta \in (0,1) $ be a constant to be determined later. Let $Z_h=Z_h(\Omega)\subset D_h\cap \Omega$ be a maximal $\theta h$-separated  set in $D_h$(a set $A$ is called   a maximal $\epsilon$-separated set in $B$ if   each two distinct points from $A$ are  at a distance greater than  $\epsilon$ and    $e(B,A)<\epsilon$). Then,
$e(D_h, Z_h)\le \theta h.$ It is easy to see that   $\{ B[z;\theta h]: z\in Z_h\}$ forms a covering of $D_h\cap \Omega$ and  the disjoint union of $B(\z;\frac{1}{2}\theta h),\, \z\in Z_h,$ is contained in $D_{2h}.$ Notice that
$\mu_d\left(B(\z;\frac{1}{2}\theta h)\right)=\nu_d(\frac{1}{2} \theta h)^d.$ Hence we have
\[
	\card(Z_h)(\frac{1}{2} \theta h)^d \nu_d \leq \mu_d(D_{2h})\rightarrow 0 
\]
and further  
	\begin{equation*}
		\card(Z_h)=o(\frac{1}{h^d})   \quad  \mbox{as} \,\, h\rightarrow 0.
	\end{equation*}
	
	For each $n\in\N,$ let $X_n\subset \Omega$  be $n$-centers for $\Omega,$  i.e.,  $\card(X_n)=n$ and  $e(\Omega,X_n)=e_n(\Omega)$.  
	Set $\xi_n^\ast:= X_{n-k_n} \cup Z_{e_n},$ where $e_n=e_n(\Omega)$ and $k_n=\card(Z_{e_n}).$  
	Then, $\card(\xi_n^\ast)=n$ and $k_n=o(\frac{1}{e_n^d})=o(n) $  from Lemma \ref{lem:e-n}.   
	Based on the definition of $ Z_{e_n},$ we have
	\[e(\partial \Omega,\xi_n^\ast)\leq e(D_{e_n}, Z_{e_n})\leq   \theta e_n(\Omega)\leq \theta e(\Omega,\xi_n^\ast) \]
	and
	\[e_n\leq e(\Omega,\xi_n^\ast)\leq e_{n-k_n}=e_n(1+o(1)) \quad  \mbox{as} \,\, n\rightarrow \infty. \]

Since $G$ is incerasing and due to (\ref{asy G}), it follows that  
	\[  G(  e (\partial \Omega,\xi_n^\ast)  ) \leq G( \theta  e ( \Omega,\xi_n^\ast)  )   =\frac{\theta^2  e ( \Omega,\xi_n^\ast)^2}{2} (1+o(1) )  \]
and 	\[    G(  e (  \Omega,\xi_n^\ast)  ) - 2G\left(  g^{-1}\left(\frac{g(e (  \Omega,\xi_n^\ast))}{2}\right)  \right)   =\frac{ e (  \Omega,\xi_n^\ast)^2}{4} (1+o(1) ) \quad  \mbox{as} \,\, n\to \infty.  \]
By comparing the two results above, if	we fix an arbitrary constant $0<\theta<1/\sqrt 2,$  then for sufficiently large $n$, it holds    that 
	\[  G(  e (\partial \Omega,\xi_n^\ast)  )       \leq G(  e (  \Omega,\xi_n^\ast)  ) - 2G\left(  g^{-1}\left(\frac{g(e (  \Omega,\xi_n^\ast))}{2}\right)  \right).  \]

Secondly, according to Theorem \ref{thm:upper-recovery}, Lemma \ref{lem:e-n} and (\ref{asy G}), it follows that 
	\begin{align*}
		R_n (W^{P(D)}_\infty(\Omega)) & \leq R(W^{P(D)}_\infty(\Omega), \xi_n^\ast) \leq  G(  e (  \Omega,\xi_n^\ast)  ) - 2G\left(  g^{-1}\left(\frac{g(e (  \Omega,\xi_n^\ast))}{2}\right)  \right)\\   &=\frac{ e (  \Omega,\xi_n^\ast)^2}{4} (1+o(1) ) 
		 = \frac{1}{4 }\left(\frac{\dens(d) \mu_d(\Omega)}{v_d n
		    }  \right)^{2/d}(1+o(1)) \quad  \mbox{as} \,\, n\to \infty. 
	\end{align*}
The proof is complete.	
\end{proof}

\section{Proof of Exact Result for Optimal Recovery}

\begin{lemma}\label{lem:uni-multi}
	For the second-order differential operator $P(D)=D^2+q$   with constant coiefficents $q\in \R$, let $\tilde h$ be  defined by (\ref{eq-h0}). Then
	$f(\cdot)=\tilde h(|\cdot-\x_0|)$ belongs to $W^{P(D)}_{\infty}(\R^d),$
	and the support of $f$ is $ B[\x_0; a].$  
\end{lemma}

\begin{proof}
	It suffices to prove the lemma for $\x_0=\mathbf 0.$ Let $f(\x):=\tilde  h(|\x|)$ for $\x\in \R^d.$     Since $\tilde  h \in C^1[0,+\infty)$ and $ \tilde  h'(0)=0,$  the radial function $f(\x):= \tilde  h(|\x|) $ belongs to $ C^1(\R^d)$ and  is twice continuously differentiable at $\x$ with  $|\x|\neq 0, t_0, a.$    
	
	For each $\x $ with $|\x|\neq 0,t_0,a,$    let $\hat{\x}=\frac
	{\x}{|\x|}$ be the unit vector  in the direction of $\x $.  For an arbitrary unit vector $\mathbf u,$ let
	$\mathbf u=\lambda \hat{\x}+\mu \hat{\x}^{\perp}$  be the
	orthogonal decomposition of $\mathbf u$ in the direction of
	$\hat\x$, where $\hat{\x}^{\perp}$ is a unit vector and  $\lambda, \mu$ satisfy $\lambda^2+\mu^2=1$. It is straightforward to verify that for $\x $ with $|\x|\neq 0,t_0,a,$
	\[  \frac{\partial^r }{\partial \mathbf
		u^r } f(\x)= (\lambda \frac{\partial }{\partial \hat\x} + \mu
	\frac{\partial }{\partial \hat\x^\perp} )^r  f(x)=\lambda^r 
	\frac{\partial^r }{\partial \hat\x^r} f(\x)=\lambda^r 
	\tilde  h^{(r)}(|\x|),  \quad r=1,2. 
	\]

	Hence, for $P(D)= D^2 + pD +q,$ 
	\[ P(\frac{\partial}{\partial \mathbf u }) f (\x)= \lambda^2 \frac{d^2}{dt^2} \tilde h(|\x|) + \lambda p \frac{d }{dt }   \tilde h(|\x |) + q  \tilde h(|\x|). \]
	
	For $P(D)= D^2 + q,$ 
	\begin{align*}
		\left|P(\frac{\partial}{\partial \mathbf u }) f (\x) \right| &\leq \left| \lambda^2 \frac{d^2}{dt^2}  \tilde h(|\x|)+ \lambda^2 q   \tilde h(|\x|) \right| +   \left|(1-\lambda^2) q   \tilde h(|\x|)\right|\\
		&\leq  \lambda^2 \left| P(\frac{d }{dt })  \tilde h(|\x|) \right| +(1-\lambda^2)  \left|q   \tilde h(|\x|)   \right|. 
	\end{align*}   
	Since $\tilde h\in W^{P(D)}_\infty [0,a],$ $| P(\frac{d }{dt })  \tilde h(|\x|) |\leq 1.$ We will prove $ |q  \tilde h(|\x|)   |\leq 1 $ in the following, then $|P(\frac{\partial}{\partial \mathbf u }) f (\x) |\leq 1$ for all  $\x $ with $|\x|\neq 0,t_0, a.$ Hence, $f\in W^{P(D)}_{\infty}(\R^d)$  and $\mbox{supp}(f)=B_d[\x_0; a].$   
	
	To obtain $ |q  \tilde h(|\x|)   |\leq 1 $, the explicit form of $ \tilde h(\x)$ is needed. 
	
	For $P(D)=D^2,$  it is obvious. 	
	For $P(D)=D^2+\beta^2$ and $D^2-\beta^2$ with $\beta>0,$   the function $\tilde h(t)$  is given in  (\ref{eq-h0-2}) and (\ref{eq-h0-3}) respectivly. It can be verified directly by taking the derivative  that $\tilde h$ is decreasing on $[0,a]$ in both cases.   Therefore, it holds  that $ \max |q  \tilde h(|\x|)   |=|q  \tilde h(0)   | \leq 1$  from $\sinh\beta a=2\sinh\beta t_0 $ and $\sin \beta a=2\sin \beta t_0 $ respectively.  
\end{proof}

\begin{theorem}\label{thm:exact-recovery}
	Let  $\Omega \subset \R^d$ be a bounded and convex body, and let $\xi\subset \Omega$ be  a finite set of nodes such that the distance from $\Omega$ to $\xi$ satisfies $e(\Omega,\xi)<\delta,$ where $\delta$ is defined in (\ref{eq-delta}).     Let $P(D)= D^2+q$ with  $q\in \R.$   Then the optimal error  is give by     
	\begin{equation*}
		R(W^{P(D)}_\infty\left(\Omega), \xi\right) =  G \left(e( \Omega,\xi)\right)-2	G\left(   g^{-1}\left(\frac{g(e (  \Omega,\xi))}{2}\right)   \right), 	  
	\end{equation*}
where functions $g$ and $G$  are defined in (\ref{eq-g}) and (\ref{eq-G}), respectively.
	
\end{theorem}
\begin{proof}
	Let $\z$ be an arbitrary point in $\Omega\setminus \xi,$ let  $a:=e(\z,\xi),$ and define $f_\z(\x):= \tilde  h(|\x-\z|)$ where  $\tilde h$  is defined in   (\ref{eq-h0}).
	By Lemma \ref{lem:uni-multi}, we have $f_\z\in W^{P(D)}_\infty(\Omega)$,  $I_\xi(f_\z)=0$, and
	 \[  \|f_\z\|_\Omega=|\tilde h (0)|=  G (e( \z,\xi))-2	G\left(   g^{-1}\left(\frac{g(e (  \z,\xi))}{2}\right)   \right).\]
	 Furthermore, by (\ref{ext-0}),
	\begin{align*} 
		R(W^{P(D)}_{\infty}(\Omega), \xi)=\sup_{\substack{f\in W^{P(D)}_{\infty}(\Omega);\\ I_\xi(f)=0  }}\|f\|_\Omega\geq  \|f_\z\|_\Omega=   G (e( \z,\xi))-2	G\left(   g^{-1}\left(\frac{g(e (  \z,\xi))}{2}\right)   \right). 	  
	\end{align*}
	Taking the supremum over $\z \in \Omega$ on the right-hand side of the above inequality, we obtain
	\begin{align*}
		R(W^{P(D)}_\infty(\Omega), \xi) \geq  G (e( \Omega,\xi))-2	G\left(   g^{-1}\left(\frac{g(e (  \Omega,\xi))}{2}\right)   \right). 	  
	\end{align*}
	Combining this with Theorem \ref{thm:upper-recovery}, we complete the proof.
\end{proof}

In the case of  $P(D)=D^2,$ we have $G(t)=\frac{t^2}{2} $ and  $t_0=\frac{a}{2}.$  The  theorem asserts  
\[ R(W^{D^2}_\infty(\Omega), \xi)  
=   \frac{1}{4}e^2( \Omega,\xi),  
\]
provided that $ e(\partial \Omega,\xi) \leq \frac{1}{\sqrt 2} e( \Omega,\xi),$  
which is a result derived from Theorem 1 in \cite{MR2735421}. 
For the case of $P(D)=D^2-\beta^2,$     
\begin{equation*}
	R(W^{D^2-\beta^2}_\infty(\Omega), \xi)  
	=  \frac{1}{\beta^2}  \left( 1 +   \cosh \beta e(\Omega,\xi) - \sqrt{  \cosh^2 \beta e(\Omega,\xi)  +3  } \right),
\end{equation*} 
as stated in  Theorem 1 of \cite{2018Optimal}. 
The case of $P(D)=D^2+\beta^2$ is a new finding.

We now proceed to the proof of Theorem \ref{thm:n recovery exact}.
\begin{proof}	
According to Theorem \ref{thm:exact-recovery}, for  any  $\xi$ with $\card(\xi)=n$, it holds  that 
	\begin{equation*}
		R(W^{P(D)}_\infty(\Omega), \xi) \geq  G (e_n( \Omega ))-2	G\left(   g^{-1}\left(\frac{g(e_n (  \Omega))}{2}\right)   \right).  	  
	\end{equation*}
This inequality   also relies on the fact that  $e_n( \Omega )\leq e( \Omega,\xi)$ and $G(a) -  2G( g^{-1}( \frac{g(a)}{2}  )  )$ is inceasing with respect to $a$.
 
	By  taking the infimum over $\xi$ with $\card(\xi)\leq n,$  we obtain 
	\[   	R_n(W^{P(D)}_\infty(\Omega) )  \geq  G (e_n( \Omega ))-2	G\left(   g^{-1}\left(\frac{g(e_n (  \Omega))}{2}\right)   \right). \]
	Letting $n$ tend  to infinity,  and  applying Lemma \ref{lem:e-n} and (\ref{asy G}), we find that
	\[ 	R_n(W^{P(D)}_\infty(\Omega) )  \geq \frac{1}{4 }\left(\frac{\dens(d) \mu_d(\Omega)}{v_d n
	}  \right)^{2/d}(1+o(1)) \quad  \mbox{as} \,\, n\to \infty. \] 
	Combining this lower bound   with the upper bound provided in  Theorem \ref{thm:n recovery upper}, we complete the proof. 	
\end{proof}
We conjecture that Theorem \ref{thm:n recovery exact} also applies to general differential operators of the form $P(D)= D^2 + pD +q$. However,   Lemma \ref{lem:uni-multi}  relies on the specific and simple expressions of $t_0$ and the function $\tilde h$, which are unknown in the general case.

\def\cprime{$'$}

\end{document}